\documentclass[graybox]{svmult}

\usepackage{mathptmx}       
\usepackage{helvet}         
\usepackage{courier}        
\usepackage{type1cm}        
\usepackage{makeidx}         
\usepackage{graphicx}        
\graphicspath{{Figures/}}
\usepackage{multicol}        
\usepackage[bottom]{footmisc}

\makeindex             

\usepackage{amsmath,amssymb}
\usepackage{algorithm}
\usepackage{algorithmic}
\usepackage{caption}
\usepackage{subcaption}
\newcommand\restr[2]{\ensuremath{\left.#1\right|_{#2}}}
\newcommand{\Tau}{\mathcal{T}}
\newcommand{\parset}{\mathcal{D}}
\newtheorem{prop}{Proposition}[section]

\begin{document}

\title*{Reduced basis methods for quasilinear elliptic PDEs with applications to permanent magnet synchronous motors}
\author{Michael Hinze, Denis Korolev}
\institute{Michael Hinze \at University of Koblenz-Landau, Mathematical Institute, \email{hinze@uni-koblenz.de}
\and Denis Korolev \at University of Koblenz-Landau, Mathematical Institute, \email{korolev@uni-koblenz.de}}
%
%
\maketitle

\abstract*{ In this paper, we propose a certified reduced basis (RB) method for \textit{quasilinear} elliptic problems together with  its application to \textit{nonlinear magnetostatics} equations, where the later model  permanent magnet synchronous motors (PMSM). The parametrization enters through the geometry of the domain and thus, combined with the nonlinearity, drives our reduction problem.  We provide a residual-based \textit{a-posteriori} error bound which, together with the Greedy approach, allows to construct  reduced-basis spaces of small dimensions. We use the empirical interpolation method (EIM) to guarantee the efficient \textit{offline-online} computational procedure. The reduced-basis solution is then obtained with the surrogate of the Newton's method. The numerical results indicate that the proposed reduced-basis method provides a significant computational gain, compared to a finite element method.}

\abstract{ In this paper, we propose a certified reduced basis (RB) method for \textit{quasilinear} elliptic problems together with  its application to \textit{nonlinear magnetostatics} equations, where the later model  permanent magnet synchronous motors (PMSM). The parametrization enters through the geometry of the domain and thus, combined with the nonlinearity, drives our reduction problem.  We provide a residual-based \textit{a-posteriori} error bound which, together with the Greedy approach, allows to construct  reduced-basis spaces of small dimensions. We use the empirical interpolation method (EIM) to guarantee the efficient \textit{offline-online} computational procedure. The reduced-basis solution is then obtained with the surrogate of the Newton's method. The numerical results indicate that the proposed reduced-basis method provides a significant computational gain, compared to a finite element method.}

\section{Introduction}
\label{sec:1}
A crucial task in the design of electric motors is the creation of proper magnetic circuits. In permanent magnet electric motors, the latter is created by electromagnets and permanent magnets. The corresponding mathematical model is governed by a quasilinear elliptic PDE (magnetostatic approximation of Maxwell equations) which describes the magnetic field generated by the sources. One of the engineering design goals consists in improving the performance of the motor through modifying the size and/or location of the permanent magnets. This problem can be viewed as a parameter optimization problem \cite{alla2019certified,bontinck2018robust,ion2018robust,lass2017model}, where the parameters determine the geometry of the computational domain. The underlying optimization problem then requires repeated solutions of the nonlinear (in general) elliptic problem on the parametrized domain.  Therefore, there is an increasing demand for the fast and reliable reduced models as surrogates in the optimization problem. To achieve this goal we use the reduced-basis method \cite{haasdonk2017reduced,quarteroni2015reduced}. The extension of reduced-basis techniques to nonlinear problems is a non-trivial task and the crucial ingredients of the method then highly depend on the underlying problem. Efficient implementation of the greedy procedure requires  a-posteriori error bounds, which, to the best of our knowledge, are not yet available for the problem we consider. In \cite{AbdHom2015} the reduced-basis method is applied to approximate the micro-problems in a homogenization procedure for quasilinear elliptic PDEs with non-monotone nonlinearity. However, we note that this different from our approach, where we use the reduced-basis method for the approximation of the solution of a quasilinear PDE. In our case, the monotonicity of the problem allows the a-posteriori control of the global reduced-basis approximation error.  We provide the corresponding error bound for quasilinear elliptic equations, which is based on a monotonicity argument and can be viewed as a generalisation of the classical error bound for linear elliptic problems \cite{rozza2007reduced}, where the coercivity constant is now substituted by the monotonicity constant of the spatial differential operator. The computational efficiency of the reduced-basis method is based on the so-called offline-online decomposition. The offline phase corresponds to the construction of the surrogate model and depends on high-dimensional simulations, and thus is expensive. The online phase, where the surrogate model is operated, is usually decoupled from high-dimensional simulations and thus in general is inexpensive. This splitting is feasible if all the quantities in the problem admit e.g. the affine decomposition, which essentially means that all parameter dependencies can be separated from the spatial variables. The recovery of the affine decomposition in the presence of nonlinearities represents an additional challenge and it  usually is treated with the empirical interpolation method (EIM) \cite{barrault2004empirical,grepl2007efficient}. The EIM algorithm requires additional data, i.e. the basis for interpolation is constructed from nonlinearity snapshots in the ``truth" space. For the efficient numerical solution of the reduced-basis problem with Newton's method we extend the computational machinery, proposed in \cite{grepl2007efficient} for semilinear PDEs. It leads to a reduced numerical scheme with full affine decomposition and thus to a considerable acceleration in the online phase, compared to the original finite element simulations. 

\section{The quasilinear parametric elliptic PDE}
\label{sec:2}
\subsection{Abstract formulation}
\label{subsec:2}

We start by introducing the model  for a permanent magnet synchronous machine. We consider a three-phase 6-pole permanent magnet synchronous machine (PMSM) with one buried permanent magnet per pole. We parametrize the problem through the size of the magnet by introducing a three dimensional parameter $p=(p_{1},p_{2},p_{3})$ which characterizes  magnet's width $p_{1}$, magnet's height $p_{2}$ and the perpendicular distance from the magnet to the rotor $p_{3}$ in mm. In fig. 1 the geometry of the problem is shown. PMSM then can be described with sufficient accuracy by the magnetostatic approximation of Maxwell's equations
\begin{align} \label{1.1.1}
    -\nabla \cdot(\nu(x,|\nabla u (p)|)\nabla u(p))=J_{e}-\frac{\partial}{\partial x_{2}}H_{pm,1}(p)+\frac{\partial}{\partial x_{1}}H_{pm,2}(p) \quad \text{in} \quad \Omega(p)
\end{align}
with boundary conditions 
\begin{align*}
\restr{u}{BC}=\restr{u}{DA}=0 \quad \text{and } \quad \restr{u}{AB}=-\restr{u}{CD}.
\end{align*}
Here $AB,BC,CD,DA$ represent parts of the boundary $\partial \Omega$ and marked in Fig.1. We assume that $\Omega(\mu)$ represents the cross-section of the electric motor which is located in the $x_{1}-x_{2}$ plane of $\mathbb{R}^{3}$ and the solution $u$ is the $x_{3}$-component of the magnetic vector potential. The $x_{3}$-component of the current density is represented by $J_{e}$, and $H_{pm,1}(p)$ and $H_{pm,2}(p)$ are components of the permanent magnet magnetic field. The nonlinear magnetic reluctivity function
\begin{align}\label{1.1.2}
  \nu(x,\eta) =
  \begin{cases}
                                   \nu_{1}(\eta), \  \text{for $x \in \Omega^{1}(p)$} \\
                                   \nu_{2}(x),  \  \text{for $x \in \Omega^{2}(p)$}, \\
  \end{cases}
\end{align}
represents ferromagnetic properties of the material. Here we split the domain $\Omega(p)$ into two non-overlapping subdomains $\Omega^{1}(p)$ (ferromagnetic steel) and $\Omega^{2}(p)$ (air, magnet, coils) such that  $\nu_{1} \in C^{1}(\Omega^{1}(p))$ and $\nu_{2}$ is piecewise constant on $\Omega^{2}(p)$ (i.e. constant for each material). In practice, we reconstruct $\nu_{1}$ from the real $B-H$ measurements of PMSM by using cubic spline interpolation. The scheme preserves desired physical properties of the reluctivity function (see, e.g. \cite{heise1994analysis} for the details of the interpolation scheme) and provides the fast-growing nonlinearity of exponential type. We use physical constants for $\nu_{2}$. Then the reluctivity function satisfies 
\begin{align}
   0 < \nu_{\text{LB}} \leq \nu(x,\eta)\leq \nu_{0}, \quad  \forall x \in \Omega(p),
\end{align}
where $\nu_{\text{LB}}$ can be chosen independently of the parameter $p$ (see section \ref{subsec:5} for details). 

We continue with  an abstract formulation of a two-dimensional nonlinear magnetostatic field problem with geometric parametrisation, where the parameter set is given by $\parset \subset \mathbb{R}^{3}$ and describes the geometry of the permanent magnet. The regular, bounded and $p$-dependent domain $\Omega(p) \subset \mathbb{R}^{2}$ gives rise to a $p$-dependent real and separable Hilbert space $X(p):=X(\Omega(p))$  and the corresponding dual space $X'(p):=X'(\Omega(p))$. The function space $X(p)$ is such that 
\begin{align*}
   X(p):=\{ v | \ v \in L^{2}(p), \nabla v \in (L^{2}(p))^{2}, \ \restr{u}{BC}=\restr{u}{DA}=0 \ \text{and} \ \restr{u}{AB}=-\restr{u}{CD} \} 
\end{align*}
with $H_{0}^1(p) \subset X(p) \subset H^1(p)$, where $H^1(p):=\{ v | \ v \in L^{2}(p), \nabla v \in (L^{2}(p))^{2} \}$, $H_{0}^1(p):=\{ v | \ v \in H^1(p), \restr{v}{\partial \Omega}=0 \}$. The inner product on $X(p)$ is defined by $(w,v)_{X(p)}=\int_{\Omega(p)} \nabla w \cdot \nabla v \ dx$   and the induced norm is given by $\lVert v\rVert_{X(p)}=(v,v)^{1/2}_{X(p)}$, which is indeed a norm due to Poincare-Friedrichs inequality. Then the abstract problem reads as follows: for $p \in \parset$, find $u(p) \in X(p)$ satisfies
\begin{align}\label{1.1.3}
    a[u(p)](u(p),v;p)=f(v,p), \quad \forall v \in X(p),
\end{align}
where we have
\begin{align}\label{1.1.4}
a[u](w,v;p)&=\int_{\Omega(p)} \nu(x,|\nabla u|)\nabla w \cdot \nabla v\ dx, \\ f(v;p)&=\int_{\Omega(p)}(J_{e}v-H_{pm,2}\frac{\partial v}{\partial x_{1}}+H_{pm,1}\frac{\partial v}{\partial x_{2}})dx.
\end{align}
The quasilinear form $a[\cdot](\cdot,\cdot;p)$ is strongly monotone on $X(p)$ with monotonicity constant $\nu_{\text{LB}}>0$, i.e.
\begin{align}\label{1.1.5}
    a[v](v,v-w;p)-a[w](w,v-w;p) \geq \nu_{\text{LB}} \lVert v-w \rVert_{X(p)}^{2}  \quad  \forall \, v,w\in X(p),
\end{align}
and Lipschitz continuous on $X(p)$ with Lipschitz constant $3\nu_{0}>0$, i.e.
\begin{align}\label{1.1.6}
|a[u](u,v;p)-a[w](w,v;p)|\leq 3\nu_{0} \lVert u-w \rVert_{X(p)} \lVert v \rVert_{X(p)} \quad  \forall \, u,w,v\in X(p). 
\end{align}
The conditions \eqref{1.1.5}, \eqref{1.1.6}  are established, e.g. in \cite{heise1994analysis}. Then problem \eqref{1.1.3} admits a unique solution (see \cite{zeidler2013nonlinear}, Th 25.B). Moreover, those properties will be needed for the error estimates.

In order to avoid domain re-meshing caused by the change of the parameters, we transfer the domain $\Omega(p)$ to a fixed domain $\hat{\Omega}:=\Omega(\hat{p})$, where $\hat{p}$ is the reference parameter with $\hat{x}:=x(\hat{p})$ as a spatial coordinate on $\hat{\Omega}$ (see e.g. \cite{rozza2007reduced}).  Further we assume that $\hat{\Omega}=\hat{\Omega}^{1} \cup \hat{\Omega}^{2}$ and this can be decomposed into $L=L_{1}+L_{2}$ (in our case $L=12$) non-overlapping triangles (see Fig.1) so that $\hat{\Omega}=\cup_{d=1}^{L} \hat{\Omega}_{d}$ and in particular $\hat{\Omega}^{1}=\cup_{d=1}^{L_{1}} \hat{\Omega}^{1}_{d}$ and $\hat{\Omega}^{2}=\cup_{d=1}^{L_{2}} \hat{\Omega}^{2}_{d}$. The transformation $\Tau(p)$ on each triangle is affine, whereas piecewise-affine and continuous over the whole domain according to:
\begin{align}\label{1.1.7}
   \restr{\Tau(p)}{\hat{\Omega}_{d}}: \hat{\Omega}_{d} &\rightarrow \Omega(p) \\ \nonumber
    \hat{x} &\mapsto C_{d}(p)\hat{x}+z_{d}(p),
\end{align}
for $d=1,...,L$, where $C_{d}(p) \in \mathbb{R}^{2\times 2}$ and $z_{d}(p)\in \mathbb{R}^{2}$. According to \eqref{1.1.7}, the Jacobian matrix $J_{\Tau}(p)$ of the transformation $\Tau(p)$ is constant on each region of the given parametrisation, i.e. we have $\restr{J_{\Tau}(p)}{\hat{\Omega}_{d}}=C_{d}(p)$. 
\begin{figure}
    \centering
    \includegraphics[scale=0.35]{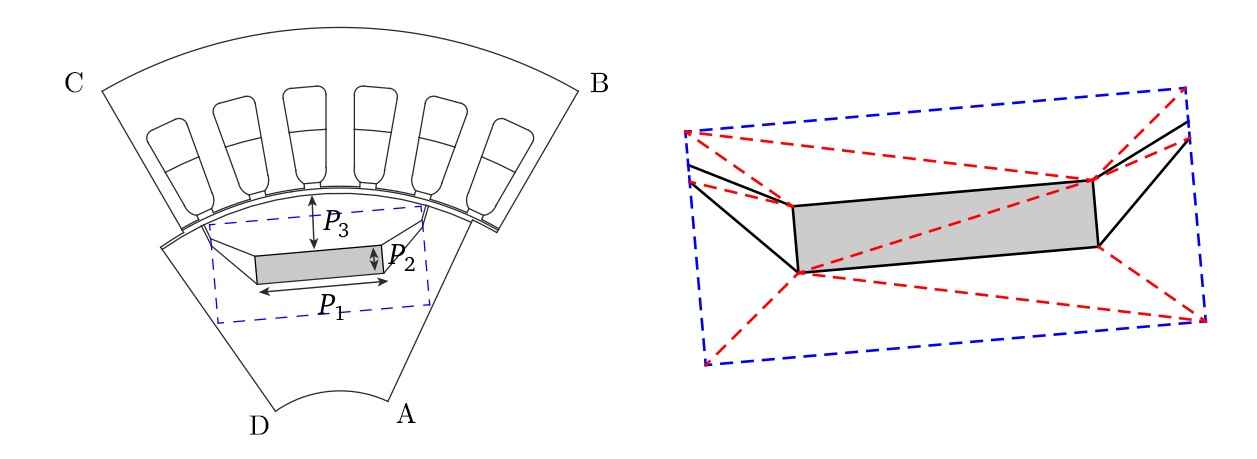}
    \caption{The cross-section of one pole of the machine with the magnet depicted in gray  
    and the region of the geometric parametrisation indicated by the dashed box. The dashed lines indicate the triangulation into $L$ triangles. Figure is adapted from \cite{bontinck2018robust}.}
    \label{fig:my_label}
\end{figure}

Now we state the problem \eqref{1.1.3} on the reference domain $\hat{\Omega}$ with the corresponding Hilbert space $\hat{X}:=X(\hat{p})$ equipped with the inner product $(\hat{w},\hat{v})_{\hat{X}}=\int_{\hat{\Omega}}\nabla \hat{w}\cdot \nabla \hat{v}d\hat{x}$ and the induced norm $\lVert \hat{v}\rVert_{\hat{X}}=(\hat{v},\hat{v})^{1/2}_{\hat{X}}$. It reads as follows: for $p \in \parset$, find $\hat{u}(p) \in \hat{X}$ so that
\begin{align}\label{1.1.8}
    a[\hat{u}(p)](\hat{u}(p),\hat{v};p)=f(\hat{v},p), \quad \forall \hat{v} \in \hat{X},
\end{align}
where the quasilinear form in \eqref{1.1.4} is now transformed with the change of variables formula into 
\begin{align}\label{1.1.9}
a[\hat{u}](\hat{w},\hat{v};p)=\int_{\hat{\Omega}}\nu(\hat{x},|J_{\Tau}^{-T}(p) \nabla \hat{u}|)[J_{\Tau}^{-T}(p)\nabla \hat{w}] \cdot [J_{\Tau}^{-T}(p)\nabla \hat{v}]\lvert \det J_{\Tau}(p) \rvert \ d\hat{x}.
\end{align}
Similarly, the linear form in \eqref{1.1.4} is transformed into
\begin{align}\label{1.1.10}
    f(\hat{v};p)=\int_{\hat{\Omega}}[f\circ \Tau(p)]\hat{v} \lvert \det J_{\Tau}(p) \rvert \ d\hat{x}.
\end{align}
Since $\hat{\Omega}=\hat{\Omega}^{1}\cup \hat{\Omega}^{2}$, we have the decomposition
\begin{align}\label{1.1.11}
    a[\hat{w}](\hat{w},\hat{v};p):=a^{\nu_{1}}[\hat{w}](\hat{w},\hat{v};p)+a^{\nu_{2}}(\hat{w},\hat{v};p), 
\end{align}
where $a^{\nu_{1}}$ is the restriction of \eqref{1.1.9} to $\hat{\Omega}^{1}$ with nonlinear reluctivity function $\nu_{1}$, and $a^{\nu_{2}}$ is the restriction of \eqref{1.1.9} to $\hat{\Omega}^{1}$ with piecewise constant reluctivity function $\nu_{2}$.  Application of Newton's method requires the computation of the derivative of $a^{\nu_{1}}$, which is given by
 \begin{align}\label{eq 1.1.12}
    a'[u](w,v;p) =& \int_{\Omega^{1}(p)}\frac{\nu'_{1}(|\nabla u|)}{\lvert \nabla u \rvert} ( \nabla u \cdot  \nabla w)(\nabla u \cdot \nabla v)dx+a^{\nu_{1}}(w,v;p)
\end{align}
and transformed as in \eqref{1.1.9} to the reference domain $\hat{\Omega}^{1}$ with the change of variables formula.

We then introduce a high dimensional finite element discretization (``truth" approximation) of our problem in the space $\hat{X}_{\mathcal{N}}=\text{span}\{\phi_{1},...,\phi_{\mathcal{N}}\} \subset \hat{X}$  of piecewise linear and continuous finite element functions. The finite element approximation is obtained by a standard Galerkin projection: given the ansatz $\hat{u}_{\mathcal{N}}(p)=\sum_{j=1}^{\mathcal{N}}\hat{u}_{\mathcal{N} \ j}(p) \phi_{j}$ for the discrete solution and testing against the basis elements in $\hat{X}_{\mathcal{N}}$ leads to the system 
\begin{align}\label{eq 1.1.13}
\sum_{j=1}^{\mathcal{N}}A^{\mathcal{N}}_{ij}(p) \hat{{u}}_{\mathcal{N} \ j}(p)=F_{\mathcal{N} \ i}(p), \quad 1\leq i \leq \mathcal{N},
\end{align}
of nonlinear algebraic equations, where $F_{\mathcal{N}}(p) \in \mathbb{R}^{\mathcal{N}},F_{\mathcal{N} \ j}(p)=f(\phi_{j};p),1\leq j \leq \mathcal{N}$ and $A^{\mathcal{N}}(p) \in \mathbb{R}^{\mathcal{N}\times \mathcal{N}}$,  $A^{\mathcal{N}}_{ij}(p)=a[\hat{u}_{\mathcal{N}}(p)](\phi_{j},\phi_{i};p), 1\leq i,j\leq \mathcal{N}$. We then apply a Newton iterative scheme: given a current iterate  $\hat{\bar{u}}_{\mathcal{N} \ j}(p),1 \leq j \leq \mathcal{N}$, we find an increment $\delta \hat{u}_{\mathcal{N} \ j}(p), 1\leq \ j \leq \mathcal{N}$, such that 
\begin{align}\label{eq 1.1.14}
    \sum_{j=1}^{\mathcal{N}}\bar{D}^{\mathcal{N}}_{ij}(p) \delta \hat{u}_{\mathcal{N} \ j}(p)=F_{\mathcal{N} \ i}(p)-\sum_{j=1}^{\mathcal{N}}\bar{A}^{\mathcal{N}}_{ij}(p) \hat{\bar{u}}_{\mathcal{N} \ j}(p), \quad 1\leq i \leq \mathcal{N},
\end{align}
where $\bar{D}^{\mathcal{N}} \in \mathbb{R}^{\mathcal{N}\times \mathcal{N}},\bar{D}^{\mathcal{N}}_{ij}(p)=a'[\hat{\bar{u}}_{\mathcal{N}}](\phi_{j},\phi_{i};p)$ and $\bar{A}^{\mathcal{N}}(p) \in \mathbb{R}^{\mathcal{N}\times \mathcal{N}}, \bar{A}^{\mathcal{N}}_{ij}(p)=a[\hat{\bar{u}}_{\mathcal{N}}](\phi_{j},\phi_{i};p),  1\leq i,j\leq \mathcal{N}$ are computed at each Newton's iteration.

From here onwards by the ``truth" solution $\hat{u}(p)$ we understand its finite element approximation $\hat{u}_{\mathcal{N}}(p)$, assuming that the given finite element approximation is good enough.

\section{Reduced basis approximation}
\subsection{An EIM-RB method}
\label{subsec:3}
To perform the reduced basis approximation, we first introduce a subset $\parset_{train}\subset \parset$ from which a sample $\parset_{N}^{u}=\{ \bar{p}_{1} \in \parset,...,\bar{p}_{N} \in \parset \}$ with associated reduced-basis space $\hat{W}_{N}^{u}=\text{span}\{\zeta_{n}:= \hat{u}(\bar{p}_{n}),\ 1\leq n \leq N \}$ of dimension $N$, which is built with the help of a weak greedy algorithm. This algorithm constructs iteratively nested (Lagrangian) spaces $\hat{W}_{n}^{u},\ 1\leq n \leq N$ using an a-posteriori error estimator $\bigtriangleup_{u}(Y;p)$, which predicts the expected approximation error for a given parameter $p$ in the space $\hat{W}_{n}^{u}=Y$. We want the expected approximation error to be less than the prescribed tolerance $\varepsilon_{RB}$. We initiate the algorithm with an arbitrary chosen parameter $\bar{p}_{1}$ with the corresponding snapshot $\hat{u}(\bar{p}_{1})$ for the basis enrichment. Next we proceed as stated in the following algorithm~\ref{alg:RB-Greedy algorithm}.
\begin{algorithm}[H]
  \caption{RB-Greedy algorithm}
  \label{alg:RB-Greedy algorithm}
  \begin{algorithmic}[1]
  \WHILE {$\varepsilon_{n} :=\underset{p \in \parset_{train}}{\max} \bigtriangleup_{u}(\hat{W}_{n}^{u},p)> \varepsilon_{RB}$}
  \STATE {$\bar{p}_{n}\leftarrow\underset{p \in \parset_{train}}{\mathrm{arg}\max}\bigtriangleup_{u}(\hat{W}_{n-1}^{u},p)$}
  \STATE{$\parset_{n}^{u}\leftarrow \parset_{n-1}^{u}\cup \{\bar{p}_{n}\}$}
  \STATE{$\hat{W}_{n}^{u}\leftarrow \hat{W}_{n-1}^{u} \bigoplus \text{span}\{\zeta_{n}\equiv \hat{u}(\bar{p}_{j})\}$}
  \STATE{$n \leftarrow n+1$}
  \ENDWHILE
  \end{algorithmic}
\end{algorithm}
\noindent
We note that the basis functions $\zeta_{n}$ are also orthonormalized relative to the $(\cdot,\cdot)_{\hat{X}}$ inner product with a Gram-Schmidt procedure to generate a well-conditioned system of equations.

The Empirical Interpolation Method (EIM) \cite{barrault2004empirical} is used to ensure the availability of offline/online decomposition in the presence of the nonlinearity. For the EIM nonlinearity approximation, we construct a sample $\parset_{M}^{\nu}=\{ p_{1}^{\nu} \in \parset,...,p_{M}^{\nu} \in \parset \}$ and associated approximation spaces $W_{M}^{\nu}=\text{span}\{\xi_{m}:= \nu(\hat{u}(p_{m}^{\nu});\hat{x};p_{m}^{\nu}),1\leq m \leq M \}=\text{span}\{q_{1},...,q_{M}\}$ together with a set of interpolation points $T_{M}=\{\hat{x}_{1}^{M},...,\hat{x}_{M}^{M}\}$. Then we build an affine approximation $\nu_{1}^{M}(\hat{u}(p);\hat{x};p)$ of $\nu_{1}(\hat{u}(p);\hat{x};p)$ as
\begin{align}\label{2.1.2}
\nu_{1}(\hat{u}(p);\hat{x};p)=\nu_{1}(\lvert J_{\Tau}^{-T}(\hat{x},p)\nabla \hat{u}(\hat{x},p) \rvert)\approx& \sum_{m=1}^{M}\varphi_{m}(p)q_{m}(\hat{x}) \\ \nonumber
=& \sum_{m=1}^{M}(B_{M}^{-1}\nu_{p})_{m}q_{m}(\hat{x}):= \nu_{1}^{M}(\hat{u}(\mu);\hat{x};p),
\end{align}
where $\nu_{p}:=\{\nu_{1}(\hat{u}(p);\hat{x}_{m}^{M};p) \}_{m=1}^{M} \in \mathbb{R}^{M}$ and $B_{M}\in \mathbb{R}^{M \times M}$ with $(B_{M})_{ij}=q_{j}(\hat{x}_{i}^{M})$ is the interpolation matrix. The EIM algorithm is initiated with an arbitrary chosen sample point $p_{1}^{\nu} \in \parset$ and then associated quantities are computed as follows
\begin{align}\label{2.1.3}
\xi_{1}=\nu(\hat{u}(p_{1}^{\nu});\hat{x};p_{1}^{\nu}),\quad \hat{x}_{1}^{M}= \underset{\hat{x}\in \hat{\Omega}}{\mathrm{arg}\sup}|\xi_{1}(\hat{x})|,\quad q_{1}=\frac{\xi_{1}}{\xi_{1}(\hat{x}_{1}^{M})}.
\end{align}
The next parameters in the sample $S_{M}^{\nu}$ are selected according to the following algorithm ~\ref{alg:EIM algorithm}. 
\begin{algorithm}[H]
  \caption{EIM algorithm}
  \label{alg:EIM algorithm}
  \begin{algorithmic}[1]
  \WHILE{$m \leq M$ and $\delta_{m}^{max}>\epsilon_{EIM}$}  
  \STATE{$ [\delta_{m}^{max},p_{m}^{\nu}] \leftarrow \underset{p \in \parset_{train}}{\mathrm{arg}\max} \underset{z \in W_{m-1}^{\nu}}{\inf} \lVert \nu_{1}(\hat{u}(p);.;p)-z \rVert_{L^{\infty}(\hat{\Omega})} $}
  \STATE{$\parset_{m}^{\nu}\leftarrow \parset_{m-1}^{\nu}\cup \{p_{m}^{\nu}\}$}
  \STATE{$r_{m}(\hat{x})=\nu_{1}(\hat{u}(p_{m}^{\nu});\hat{x};p_{m}^{\nu})-\nu_{1}^{m}(\hat{u}(p_{m}^{\nu});\hat{x};p_{m}^{\nu})$}
  \STATE{$\hat{x}_{m}^{M}=\underset{\hat{x} \in \hat{\Omega}}{\mathrm{arg}\sup}|r_{m}(\hat{x})|, \quad q_{m}=r_{m}/r_{m}(\hat{x}_{m}^{M})$}
   \STATE{$m \leftarrow m+1$}
  \ENDWHILE
  \end{algorithmic}
\end{algorithm}
\noindent

The EIM approximation of $\nu_{1}$ results in the EIM-approximation $a_{M}[\cdot](\cdot,\cdot;p)$ of the quasilinear form $a[\cdot](\cdot,\cdot;p)$ and then the reduced basis approximation is obtained by a standard Galerkin projection: given $p \in \parset$, find $\hat{u}_{N,M}(p) \in \hat{W}_{N}^{u}$ such that
\begin{align}\label{2.1.6}
    a_{M}[\hat{u}_{N,M}(p)](\hat{u}_{N,M}(p),\hat{v}_{N};p)=f(\hat{v}_{N};p), \quad \forall \hat{v}_{N} \in  \hat{W}_{N}^{u}
\end{align}
holds. Since $\hat{\Omega}=\hat{\Omega}^{1}\cup \hat{\Omega}^{2}$, we have the decomposition
\begin{align}\label{2.1.7}
    a_{M}[\hat{w}](\hat{w},\hat{v};p):=a_{M}^{\nu_{1}}[\hat{w}](\hat{w},\hat{v};p)+a^{\nu_{2}}(\hat{w},\hat{v};p),
\end{align}
where $a_{M}^{\nu_{1}}[\cdot](\cdot,\cdot;p)$ is the EIM-approximation of $a^{\nu_{1}}[\cdot](\cdot,\cdot;p)$ with nonlinear reluctivity $\nu_{1}(p)$ replaced by its EIM counterpart ${\nu}_{1}^{M}(p)$.

\subsection{Error estimation}
We define  $W_{N}^{u}(p):=\{w_{N}\ |\ w_{N}= \hat{w}_{N} \circ \Tau^{-1}, \ \hat{w}_{N} \in \hat{W}_{N}^{u}  \}$ as a push-forward reduced-basis space over the parametrised domain $\Omega(p)$ for error estimation purposes, where $\Tau^{-1}$ is the inverse of the geometric transformation \eqref{1.1.7}. First we study the convergence of $\hat{u}_{N,M}(p)\rightarrow \hat{u}(p)$. 
\begin{prop}[A-priori Error Bound]
Assume that the EIM-approximation error of the nonlinearity satisfies $\sup_{p \in \parset} \lVert \nu(p) - \nu^{M}(p) \rVert_{L^{\infty}} \leq \epsilon_{M}$. Assume further that $a(\cdot;\cdot,\cdot;p)$ is Lipschitz continuous on $X(p)$ with Lipschitz constant $3\nu_{0}>0$  and that the EIM-approximation $a_{M}(\cdot;\cdot,\cdot;p)$ of $a(\cdot;\cdot,\cdot;p)$ is strongly monotone with monotonicity constant $\tilde{\nu}_{\text{LB}}:=\nu_{\text{LB}}-\epsilon_{a}>0$. Then we have
 \begin{align}
 \lVert \hat{u}(p)-\hat{u}_{N,M}(p)\rVert_{\hat{X}} \leq \sqrt{\frac{C_{2}(p)}{C_{1}(p)}}\inf_{\hat{w}_{N}\in \hat{W}_{N}^{u}}\{\left( 1+\frac{3\nu_{0}}{\tilde{\nu}_{\text{LB}}}\right)\lVert \hat{u}(p)-\hat{w}_{N}\rVert_{\hat{X}}+\frac{\epsilon_{M}}{\tilde{\nu}_\text{LB}} \lVert \hat{w}_{N} \rVert_{\hat{X}}\} 
 \end{align}
 with the geometric constants
\begin{align}\label{const1}
    C_{1}(p):= \underset{1\leq d\leq L}{\min}\{\lambda_{min}(C_{d}(p)^{-1}C_{d}(p)^{-T})|\det C_{d}(p)|\}
\end{align}
and
\begin{align}\label{const2}
    C_{2}(p):= \underset{1\leq d\leq L}{\max}\{\lambda_{max}(C_{d}(p)^{-1}C_{d}(p)^{-T})|\det C_{d}(p)|\}
\end{align}
\end{prop}

\begin{proof} 
\smartqed
Set $u:=u(p) \in X(p),\ u_{N,M}:=u_{N,M}(p) \in W_{N}^{u}(p)$ and let $w_{N} \in W_{N}^{u}(p)$ be arbitrary. We use the strong monotonicity condition and Lipschitz continuity to obtain the bound 
\begin{align*}
 \tilde{\nu}_{\text{LB}} \lVert u_{N,M}- w_{N} \rVert_{X(p)}^{2} \leq & \ a_{M}[u_{N,M}](u_{N,M},u_{N,M}-w_{N})-a_{M}[w_{N}](w_{N},u_{N,M}-w_{N}) \\
\leq  & \ a[u](u,u_{N,M}-w_{N})-a[w_{N}](w_{N},u_{N,M}-w_{N}) \\  + & \ |a[w_{N}](w_{N},u_{N,M}-w_{N})-a_{M}[w_{N}](w_{N},u_{N,M}-w_{N})| \\
\leq & \ 3\nu_{0}\lVert u- w_{N} \rVert_{X(p)}\lVert u_{N,M}- w_{N} \rVert_{X(p)}\\
+& \sup_{p \in \parset} \lVert \nu(p) - \nu^{M}(p) \rVert_{L^{\infty}}\lVert w_{N} \rVert_{X(p)}\lVert u_{N,M}-w_{N} \rVert_{X(p)}.
\end{align*}
Dividing both sides by $\tilde{\nu}_\text{LB}\lVert u_{N,M}-w_{N} \rVert_{X(p)}$ and using the triangle inequality 
\begin{align*}
    \lVert u-u_{N,M}\rVert_{X(p)} \leq \lVert u-w_{N}\rVert_{X(p)}+\lVert u_{N,M}-w_{N}\rVert_{X(p)},
\end{align*}
we obtain the estimate
\begin{align}\label{3.2.2}
 \lVert u(p)-u_{N,M}(p)\rVert_{X(p)} \leq \left( 1+\frac{3\nu_{0}}{\tilde{\nu}_{\text{LB}}}\right)\lVert u(p)-w_{N}\rVert_{X(p)}+\frac{\epsilon_{M}}{\tilde{\nu}_\text{LB}} \lVert w_{N} \rVert_{X(p)}. 
 \end{align}
Inspecting the geometric dependence with the lower bound
 \begin{align}\label{3.2.3}
    \lVert v \rVert_{X(p)}^{2}=&\sum_{d=1}^{L}\sum_{i,j=1}^{2}[C_{d}(p)^{-1}C_{d}(p)^{-T}]_{ij}|\det C_{d}(p)| \int_{\hat{\Omega}_{d}}\frac{\partial \hat{v}}{\partial \hat{x}_{i}}\frac{\partial \hat{v}}{\partial \hat{x}_{j}}\ d\hat{x}\\ \nonumber
    \geqslant& \underset{1\leq d\leq L}{\min}\{\lambda_{min}(C_{d}(p)^{-1}C_{d}(p)^{-T})|\det C_{d}(p)|\}  \ \lVert \hat{v} \rVert_{\hat{X}}^{2}=C_{1}(p)\lVert \hat{v} \rVert_{\hat{X}}^{2},
\end{align}
applied to the left-hand side of \eqref{3.2.2}, together with the similarly established upper bound
 \begin{align} \label{3.2.4}
    \lVert v \rVert_{X(p)}^{2}\leq  \underset{1\leq d\leq L}{\max}\{\lambda_{max}(C_{d}(p)^{-1}C_{d}(p)^{-T})|\det C_{d}(p)|\} \ \lVert \hat{v} \rVert_{\hat{X}}^{2}=C_{2}(p)\lVert \hat{v} \rVert_{\hat{X}}^{2}
\end{align}
applied to the right-hand side of \eqref{3.2.2},  the desired result follows after a short calculation.
\qed 
\end{proof}

 For efficient implementation of the reduced basis methodology and the verification of the error, it is necessary to provide an a-posteriori error bound, which can be quickly evaluated.  For this we establish an error bound based on the residual. We denote by $r_{M}(\cdot;p) \in \hat{X}'$  the residual (formed on the reference domain) of the problem, defined naturally as
\begin{align}\label{3.2.5}
    r_{M}(\hat{v};p)=f(\hat{v};p)-a_{M}[\hat{u}_{N,M}](\hat{u}_{N,M},\hat{v};p).
\end{align}
 We have the following 
\begin{prop}[A-posteriori Error Bound]
 Let $\nu_{\text{LB}}>0$ be the lower bound of the monotonicity constant. Then, the RB-EIM error $\hat{e}_{N,M}(p):=\hat{u}(p)-\hat{u}_{N,M}(p)$ can be bounded by
 \begin{align}\label{3.2.6}
\lVert \hat{e}_{N,M}(p)\rVert_{\hat{X}}\leq \frac{1}{\nu_{\text{LB}}\ C_{1}(p)}(\lVert r_{M}(\cdot;p)\rVert_{\hat{X}'}+C_{2}(p)\delta_{M}(p)\lVert \hat{u}_{N,M}(p) \rVert_{\hat{X}}):=\bigtriangleup_{N,M}(p)
 \end{align}
with the geometric constants \eqref{const1}, \eqref{const2} and the EIM-approximation error
\begin{align*}
\delta_{M}(p)=\sup_{\hat{x} \in \hat{\Omega}}|\nu(\lvert J_{\Tau}^{-T}(\hat{x},p)\nabla \hat{u}(\hat{x};p) \rvert)-\nu^{M}(\lvert J_{\Tau}^{-T}(\hat{x},p)\nabla \hat{u}(\hat{x};p) \rvert)|
\end{align*}
 of the nonlinearity 
\end{prop}
\begin{proof}
\smartqed
Since in the case $e_{N,M}=0$ there is nothing to show, we assume that $e_{N,M} \neq 0$. We then use strong monotonicity condition \eqref{1.1.5} and the definition of the residual \eqref{3.2.5} to estimate
\begin{align*}
\nu_{\text{LB}} \lVert e_{N,M} \rVert_{X(p)}^{2}\leq   a[u](u,e_{N,M})-a[u_{N,M}](u_{N,M},e_{N,M}) \\
=f(e_{N,M})-a_{M}[u_{N,M}](u_{N,M},e_{N,M})+ a_{M}[u_{N,M}](u_{N,M},e_{N,M})-a[u_{N,M}](u_{N,M},e_{N,M})\\
:=r_{M}(e_{N,M})+a_{M}[u_{N,M}](u_{N,M},e_{N,M})- a[u_{N,M}](u_{N,M},e_{N,M})\\
=r_{M}(\hat{e}_{N,M})+a_{M}[u_{N,M}](u_{N,M},e_{N,M})- a[u_{N,M}](u_{N,M},e_{N,M})\\
\leq \lVert r_{M} \rVert_{\hat{X}'}\lVert \hat{e}_{N,M} \rVert_{\hat{X}}+ \delta_{M}(p) \lVert u_{N,M}\rVert_{X(p)}\lVert e_{N,M} \rVert_{X(p)} 
\end{align*}
Now the final result follows from the estimate \eqref{3.2.3} and \eqref{3.2.4}, applied to $\lVert e_{N,M} \rVert_{X(p)}^{2}$ and the right-hand side of the inequality, correspondingly.
\qed 
\end{proof}
We address the computational realization of the estimator \eqref{3.2.6} in the next section. Next we denote by $r(\cdot;p) \in \hat{X}'$  the residual of the original problem (without EIM reduction), defined as
\begin{align}\label{3.2.7}
    r(\hat{v};p)=f(\hat{v};p)-a[\hat{u}_{N}](\hat{u}_{N},\hat{v};p)
\end{align}
and let $\hat{e}_{N}(p):=\hat{u}(p)-\hat{u}_{N}(p)$ be the error of the reduced-basis approximation. Along the lines of proposition 3.2 one can prove the error bound
\begin{align}\label{3.2.8}
\lVert \hat{e}_{N}(p)\rVert_{\hat{X}}\leq \frac{\lVert r(\cdot;p)\rVert_{\hat{X}'}}{\nu_{\text{LB}}\ C_{1}(p)}:=\bigtriangleup_{N}(p).
\end{align}
We use \eqref{3.2.8} to investigate the factor of overestimation in the reduced-basis approximation. 
\begin{prop}[Effectivity bound for RB-approximation] Let $\eta_{N}(p)=\frac{\bigtriangleup_{N}(p)}{\lVert \hat{e}_{N} \rVert_{\hat{X}}}$. Then
\begin{align}\label{3.2.9}
   \eta_{N}(p) \leq \frac{3\nu_{0}}{\nu_{\text{LB}}}\sqrt{C_{1}(p)C_{2}(p)}
\end{align}

\end{prop}
\noindent
\begin{proof}
\smartqed
Let $\hat{v}_{r} \in \hat{X}$ denote the Riesz-representative of $r(\cdot;p)$. Then we have
\begin{align*}
\langle \hat{v}_{r},\hat{v}\rangle_{\hat{X}}=r(\hat{v};p), \ \hat{v} \in \hat{X}, \quad  \lVert \hat{v}_{r} \rVert_{\hat{X}}=\lVert r(\cdot;p) \rVert_{\hat{X}'}.   
\end{align*}
Now let $ v_{r}:= \hat{v}_{r} \circ \Tau^{-1} \in X(p)$. Then, using Lipshitz continuity of \eqref{1.1.6}, we have
 \begin{align*}
\lVert v_{r} \rVert_{X(p)}^{2}=\langle v_{r},v_{r} \rangle_{X(p)}=r(v_{r};\mu)&=a[u](u,v_{r};p)-a[u_{N}](u_{N},v_{r};p) \\
&\leq 3\nu_{0} \lVert e_{N} \rVert_{X(p)} \lVert v_{r} \rVert_{X(p)}.
\end{align*}
With the estimates \eqref{3.2.3} and \eqref{3.2.4}, applied to both sides of this inequality, we obtain
\begin{align*}
\frac{\lVert \hat{v}_{r} \rVert_{\hat{X}}}{\lVert \hat{e}_{N} \rVert_{\hat{X}}}\leq 3\nu_{0} \sqrt{\frac{C_{2}(p)}{C_{1}(p)}}.
\end{align*}
With \eqref{3.2.8} we then conclude
\begin{align*}
  \eta_{N}(p)=\frac{\bigtriangleup_{N}(p)}{\lVert \hat{e}_{N} \rVert_{\hat{X}}}=\frac{\lVert \hat{v}_{r} \rVert_{\hat{X}}}{\nu_{\text{LB}} \ C_{1}(p)\lVert \hat{e}_{N} \rVert_{\hat{X}}}\leq \frac{3\nu_{0}}{\nu_{\text{LB}}}\sqrt{C_{1}(p)C_{2}(p)}.
\end{align*}
and obtain the effectivity bound.
\qed 
\end{proof}
This bound is further used to explain the gap between the true error and the estimator.

\subsection{Computational procedure}
\label{subsec:4}
The computational process in the reduced basis modelling can be split into the offline and the online phase. The computations in the offline phase depend on the dimension $\mathcal{N}$ of the finite element space and are  expensive, but should be performed only once. The computations in the online phase are independent of $\mathcal{N}$, with computational complexity which depends only on the the dimension $N$ of the reduced-basis approximation space and the dimension $M$ of the EIM approximation space. The key concept utilized here is parameter-separability (or affine decomposition) of all the forms involved in the problem. With EIM we can achieve an affine decomposition of the quasilinear form 
\begin{align} \label{3.3.1}
a_{M}^{\nu_{1}}[\hat{u}_{N,M}(p)](\hat{w},\hat{v};p)&=\sum_{m=1}^{M}\sum_{d=1}^{L_{1}}\sum_{i,j=1}^{2}\varphi_{m}(p) \Phi_{d,L_{1}}^{i,j}(p)a_{m,d}^{i,j}(\hat{w},\hat{v}), \\ \nonumber a^{\nu_{2}}(\hat{w},\hat{v};p)&=\sum_{d=1}^{L_{2}}\sum_{i,j=1}^{2}\Phi_{d,L_{2}}^{i,j}(p)a_{d}^{i,j}(\hat{w},\hat{v}), 
\end{align}
such that $\Phi_{d,L_{1}}^{i,j}:\parset \rightarrow \mathbb{R}$ for $d=1,...,L_{1},i,j=1,2$ and $\Phi_{d,L_{2}}^{i,j}:\parset \rightarrow \mathbb{R}$ for $d=1,...,L_{2},i,j=1,2$ are functions depending on  $p$ and on the parameter independent forms 
\begin{align*}
 a_{m,d}^{i,j}(\hat{w},\hat{v})&=\int_{\hat{\Omega}^{1}_{d}}q_{m}\frac{\partial \hat{w}}{\partial \hat{x}_{i}}\frac{\partial \hat{v}}{\partial \hat{x}_{j}}\ d\hat{x} , \ 1\leq d \leq L_{1}, \ 1\leq i,j \leq 2 ,\\
a_{d}^{i,j}(\hat{w},\hat{v})&=\int_{\hat{\Omega}_{d}^{2}}\frac{\partial \hat{w}}{\partial \hat{x}_{i}}\frac{\partial \hat{v}}{\partial \hat{x}_{j}}\ d\hat{x}, \ 1\leq d \leq L_{2}, \ 1\leq i,j \leq 2.
\end{align*}
For notational convenience, we set $c_{m}(\hat{w},\hat{v};p):=\sum_{d=1}^{L_{1}}\sum_{i,j=1}^{2} \Phi_{d,L_{1}}^{i,j}(p)a_{m,d}^{i,j}(\hat{w},\hat{v})$, so that
\begin{align*}
 a_{M}^{\nu_{1}}[\hat{u}_{N,M}(p)](\hat{w},\hat{v};p)=\sum_{m=1}^{M} \varphi_{m}(p)c_{m}(\hat{w},\hat{v};p).   
\end{align*}
Similarly, the affine decomposition of  $f$  has the form 
\begin{align*}
    f(\hat{v};p)=\int_{\hat{\Omega}}J_{e}\hat{v}\ d\hat{x}-\sum_{d=1}^{L}\sum_{i=1}^{2}|\det C_{d}(p)|C_{d}(p)^{-T}_{1\ i}\int_{\hat{\Omega}_{d}}H_{pm,1} \frac{\partial \hat{v}}{\partial \hat{x}_{i}}d\hat{x}\\
    +\sum_{d=1}^{L} \sum_{i=1}^{2}|\det C_{d}(p)|C_{d}(p)^{-T}_{2\ i}\int_{\hat{\Omega}_{d}} H_{pm,2}\frac{\partial \hat{v}}{\partial \hat{x}_{i}}d\hat{x}=\sum_{q=1}^{Q_{f}}\Phi_{q}^{f}(p)f_{q}(\hat{v}),
\end{align*}
where $\Phi_{q}^{f}:\parset \rightarrow \mathbb{R}$ for $q=1,...,Q_{f}$ are parameter dependent functions and parameter independent forms $f_{q}(\hat{v})$.

We now give the details of the numerical scheme for the nonlinear part, defined on the domain $\hat{\Omega}^{1}$. The second term  in \eqref{3.3.1} is linear and can be treated similarly. We expand our reduced basis solution as $\hat{u}_{N,M}(p)=\sum_{j=1}^{N}\hat{u}_{N,M \ j}\zeta_{j}$ and test against the basis elements in $\hat{W}_{N}^{u}$ to obtain the algebraic equations
\begin{align}\label{3.3.2}
    \sum_{j=1}^{N}\sum_{m=1}^{M} \varphi_{m}(p)C^{j(N,M)}_{i\ m}(p)\hat{u}_{N,M\ j}(p)=F_{N\ i}(p), \quad 1\leq i \leq N,
\end{align}
where $C^{j(N,M)}(p) \in \mathbb{R}^{N\times M},C^{j(N,M)}_{i\ m}(p)=c_{m}(\zeta_{j},\zeta_{i};p),1\leq i \leq N, 1\leq m \leq M,1\leq j \leq N$, and $F_{N\ i}(p)=f(\zeta_{i};p)$. Since $\varphi_{M}(p)=\{\varphi_{M\ k}(p)\}_{k=1}^{M} \in \mathbb{R}^{M}$ is given by 
\begin{align}\label{3.3.3}
    \sum_{k=1}^{M}B^{M}_{m \ k} \varphi_{M\ k}(p)&=\nu_{1}(\hat{u}_{N,M}(\hat{x}_{m}^{M};p);\hat{x}_{m}^{M};p), \quad 1\leq m\leq M \\ \nonumber
    &=\nu_{1}(\sum_{n=1}^{N}\hat{u}_{N,M\ n}(p)\zeta_{n}(\hat{x}_{m}^{M});\hat{x}_{m}^{M};p),\quad 1\leq m\leq M.
\end{align}
We then insert \eqref{3.3.3} into \eqref{3.3.2} to get the following nonlinear algebraic equation system 
\begin{align}\label{3.3.4}
\sum_{j=1}^{N}\sum_{m=1}^{M}D^{j(N,M)}_{i\ m}(p)\nu_{1}(\sum_{n=1}^{N}\hat{u}_{N,M\ n}(p)\zeta_{n}(\hat{x}_{m}^{M});\hat{x}_{m}^{M};p)\ \hat{u}_{N,M\ j}(p)=F_{N\ i}(p), \quad 1\leq i \leq N,
\end{align}
with $D^{j(N,M)}(p)=C^{j(N,M)}(p)(B^{M})^{-1} \in \mathbb{R}^{N\times M}$.

To solve \eqref{3.3.4} for $\hat{u}_{N,M\ j}(p), 1\leq j \leq N$, we apply a Newton's iterative scheme: given the current iterate $\hat{\bar{u}}_{N,M\ j}(p), 1\leq j \leq N$, compute an increment $\delta \hat{u}_{N,M\ j}(p), 1\leq j \leq N$, from
\begin{align}\label{3.3.5}
    \sum_{j=1}^{N}[\bar{A}^{N}_{ij}(p)+\bar{E}^{N}_{ij}(p)]\delta \hat{u}_{N,M\ j}(p)=R_{N \ i}(p), \quad 1\leq i \leq N,
\end{align}
and update $\hat{\bar{u}}_{N,M\ j}(p):=\hat{\bar{u}}_{N,M\ j}(p)+\delta \hat{u}_{N,M\ j}(p)$, where the residual $R_{N}(p) \in \mathbb{R}^{N}$ for the Newton's scheme must be calculated at every Newton iteration according to
\begin{align}\label{3.3.6}
      R_{N \ i}(p)=F_{N\ i}(p)-\sum_{j=1}^{N}\sum_{m=1}^{M}D^{j(N,M)}_{i\ m}(p)\nu_{1}(\sum_{n=1}^{N}\hat{\bar{u}}_{N,M\ n}(p)\zeta_{n}(\hat{x}_{m}^{M});\hat{x}_{m}^{M};p)\ \hat{\bar{u}}_{N,M\ j}(p).
\end{align}
Furthermore $\bar{A}^{N}(p)\in \mathbb{R}^{N\times N}$, $\bar{A}_{ij}^{N}(p)=a_{M}^{\nu_{1}}[\hat{\bar{u}}_{N,M}(p)](\zeta_{j},\zeta_{i};p)$ and $\bar{E}^{N}(p) \in \mathbb{R}^{N \times N}$ with
\begin{align}\label{3.3.7}
   \bar{E}^{N}_{ij}(p)&=\sum_{s=1}^{N}\bar{\hat{u}}_{N,M\ s}(p)\sum_{m=1}^{M}D^{s(N,M)}_{i\ m}(p)\frac{g^{j}_{m}(p)}{|J_{\Tau}^{-T}(\hat{x}_{m}^{M},p)\nabla\hat{\bar{u}}_{N,M}(\hat{x}_{m}^{M};p)|},\quad 1\leq i,j\leq N, 
\end{align}
where 
\begin{align*}   
   g^{j}_{m}(p)&=\partial_{1}\nu_{1}(\hat{\bar{u}}_{N,M}(\hat{x}_{m}^{M};p);\hat{x}_{m}^{M};p)[J_{\Tau}^{-T}(\hat{x}_{m}^{M},p)\nabla\bar{\hat{u}}_{N,M}(\hat{x}_{m}^{M};p)]\cdot[J_{\Tau}^{-T}(\hat{x}_{m}^{M},p)\nabla \zeta_{j}(\hat{x}_{m}^{M})]
\end{align*}
for $1\leq m \leq M$. Here $\partial_{1}\nu_{1}$ denotes the partial derivative of $\nu_{1}(p)$ with respect to its first argument.

Although \eqref{3.3.7} looks quite involved, it possesses an affine decomposition and allows efficient assembling in the online phase. Indeed, the matrix $D^{j(N,M)}(p)$ is parameter-separable, since $C^{j(N,M)}(p)$ is parameter-separable and the evaluation of $g^{j} \in \mathbb{R}^{M}$ in \eqref{3.3.7} requires the evaluation of the reduced-basis functions only on the set of interpolation points $T_{M}$. Therefore, these quantities can be computed and stored in the offline phase and can be assembled in the online phase independently of $\mathcal{N}$. The operation count associated with each Newton's update is then as follows: the assembling of the residual $R_{N}(p)$ in \eqref{3.3.6} is achieved at cost $\mathcal{O}(MN^{2})$ together with the EIM system solve at cost $\mathcal{O}(M^{2})$. The Jacobian $\bar{A}^{N}(p)+\bar{E}^{N}(p)$ in \eqref{3.3.4} is assembled at cost $\mathcal{O}(MN^{3})$, where the dominant cost is for the assembling of $\bar{E}^{N}(p)$. It is then inverted at cost $\mathcal{O}(N^{3})$. The operation count in the online phase is thus $\mathcal{O}(MN^{3})$ per Newton iteration. However, we observe in our numerical experiment that it is sufficient to use $\bar{A}^{N}(p)$ and drop $\bar{E}^{N}(p)$ term in \eqref{3.3.5}, which results in $\mathcal{O}(MN^{2}+N^{3})$ operations per Newton iteration.

Next we address the computation of the a-posteriori error bound \eqref{3.2.6}. It requires the computation of the dual norm of the residual \eqref{3.2.5}. Since the right-hand side $f(\cdot;p)$ and $a_{M}[\cdot](\cdot,\cdot;p)$ are parameter-separable, the residual $r_{M}(\cdot;p)$ is also parameter-separable and admits an affine decomposition together with its Riesz-representative $\hat{v}_{r}(p) \in \hat{X}$ according to
\begin{align}
   r_{M}(\hat{v};p)=\sum_{q=1}^{Q_{r}} \Phi_{q}^{r}(p)r_{M \ q}(\hat{v}), \quad \hat{v}_{r}(p)=\sum_{q=1}^{Q_{r}} \Phi_{q}^{r}(\mu)\hat{v}_{r \ q},
\end{align}
where $r_{M}(\hat{v};p)=(\hat{v}_{r}(p),\hat{v})_{\hat{X}}$ for all $ \hat{v} \in \hat{X}$ and  $Q_{r}=Q_{f}+N(M+4ML_{1}+4L_{2})$. Since the dual norm of the residual is equal to the norm of its Riesz-representative, we have
\begin{align}\label{3.3.8}
 \lVert r_{M}(\cdot;p) \rVert_{\hat{X}'}=\lVert \hat{v}_{r}(p) \rVert_{\hat{X}}=(\boldsymbol{\Phi}^{r}(p)^{T}G_{r} \boldsymbol{\Phi}^{r}(p))^{1/2},   
\end{align}
where $\boldsymbol{\Phi}^{r}(p)=\{\Phi_{q}^{r}(p)\}_{q=1}^{Q_{r}} \in \mathbb{R}^{Q_{r}}$ and $G_{r} \in \mathbb{R}^{Q_{r}\times Q_{r}}$ with $(G_{r})_{ij}=(v_{r \ i},v_{r \ j})_{\hat{X}}$ and the dual norm \eqref{3.3.8} is then computed at cost $\mathcal{O}(Q_{r}^{2})$. The evaluation of the norm $\lVert \hat{u}_{N,M}(p) \rVert_{\hat{X}}$ is at cost $\mathcal{O}(N^{2})$. Once $\nu_{LB}$ is available, the constants $C_{1}(p)$ and $C_{2}(p)$ in \eqref{3.2.6} are computed directly. The EIM error is computed on the discretized domain $\hat{\Omega}_{h}\subset \hat{\Omega}$ with the reduced-basis solution 
\begin{align}\label{3.3.9}
    \delta_{M}(p)=\underset{\hat{x}_{b_{j}} \in \hat{\Omega}_{h}}{\max}\lvert \nu_{1}(\hat{u}_{N,M}(p);\hat{x}_{b_{j}};p)-\nu_{1}^{M}(\hat{u}_{N,M}(p);\hat{x}_{b_{j}};p) \rvert. 
\end{align}
The nonlinearity depends on the gradient and it is evaluated on the triangle barycenters $\hat{x}_{b_{j}}$, $1\leq j \leq \mathcal{N}_{T} $, where  $\mathcal{N}_{T}$ is the total number of triangles in the iron material region for a given finite-element triangulation. The EIM procedure results in the set of triangle barycenter points $T_{M}=\{\hat{x}_{b_{1}}^{M},...,\hat{x}_{b_{M}}^{M}\}$, where $M<<\mathcal{N}_{T}$. In the offline phase we evaluate the gradients $\{\nabla \zeta_{n}\}_{n=1}^{N}$ for each basis element $\{\zeta_{n}\}_{n=1}^{N}$ of the reduced-basis space $\hat{W}_{N}^{u}$ on the interpolation barycenters $T_{M}$. We thus store offline $\{\restr{\nabla \zeta_{n}}{T_{M}}(\hat{x}_{b_{j}}^{M})\}_{j=1}^{M}\in \mathbb{R}^{M\times 2}$ for $1\leq n \leq N$  and then efficiently evaluate the nonlinearity on $T_{M}$ with the ansatz $\restr{\nabla\hat{u}_{N,M}(p)}{T_{M}}=\sum_{j=1}^{N}\hat{u}_{N,M \ j}(p)\restr{\nabla \zeta_{j}}{T_{M}}$ online. The operation count for the EIM approximation  $\nu_{1}^{M}(\hat{u}_{N,M}(p);\hat{x}_{j};p)$ in  \eqref{3.3.9} is then $\mathcal{O}(M^{2}+\mathcal{N}_{T}M)$, and the evaluation of $\nu_{1}$ at $M$ points. We note that \eqref{3.3.9} requires the knowledge of $\nu_{1}(\hat{u}_{N,M}(p);\hat{x}_{j};p)$ and thus one full evaluation of the nonlinearity. In order to increase the online computational efficiency, an one-point estimator $\hat{\varepsilon}_{M}(p)$ can be used (see, e.g. \cite{grepl2007efficient}). It requires the evaluation of the nonlinearity at only one point, but $\hat{\varepsilon}_{M}(p)\leq \delta_{M}(p)$ in general, thus this lower bound estimator must be effective, i.e. $\frac{\hat{\varepsilon}_{M}(p)}{\delta_{M}(p)}$
should close to 1. In our case the nonlinearity is of the exponential type and the effectivity of the bound is of the order $10^2$ in practice.

\subsection{Numerical results}
\label{subsec:5}

First we introduce a parameter set $\parset=[18,19]\times[4,5] \times [7,8]$.  The nonlinear reluctivity function $\nu_{1}(p)$ is reconstructed from the real $B-H$ measurements using cubic spline interpolation. Finite element simulations are based on a mesh composed of $121012$ triangles and $60285$ nodes (excluding Dirichlet boundary nodes). Piecewise linear, continuous finite element functions are chosen for the finite element approximation. We solve the finite element problem with Newton's method. We iterate unless the norm of the residual is less than the tolerance level, which we set to $10^{-4}$. The  tolerance level $10^{-5}$ is used for the RB Newton's method. 

We generate the RB-EIM model as follows: we start from $\parset_{train}^{EIM (1)} \subset \parset$  (a regular $6 \times 6\times 6$ grid over $\parset$ of size 216) and compute finite element solutions for each parameter in $\parset_{train}^{EIM(1)}$ to approximate the nonlinearity with the EIM within the prescribed tolerance $\epsilon_{EIM}=5 \cdot 10^{-1}$. Since the norm $\lVert \hat{u}_{N,M}(p) \rVert_{\hat{X}}$ is of the order $10^{-2}$, we hope to further balance the contributions of the reduced-basis and EI nonlinearity approximation in the estimator on the test set. Next we run the RB-Greedy procedure with the prescribed tolerance $\epsilon_{RB}=10^{-2}$ for the estimator \eqref{3.2.6} on $\parset_{train} \subset \parset$, where $\parset_{train}$ is a regular $10 \times 10\times 10$ grid over $\parset$ of size 1000. We set $ \nu_{LB}=110$, since
\begin{align}\label{lower}
 \nu_{LB} \leq \underset{\hat{x} \in \hat{\Omega}}{\min} \ \nu_{1}(|\nabla \hat{u}_{N,M}(\hat{x};p)|) \simeq 110   
\end{align}
 for all $p \in \parset_{train}$ in our setting. This is a robust heuristic procedure, since for small $N$, the reduced-basis solution $\hat{u}_{N,M}(\hat{x};p)$ is a good approximation to $\hat{u}(\hat{x};p)$ in the regions with low magnetic flux density $|\nabla \hat{u}(\cdot;p)|$. The size of the magnet (change in the parameter $p$) influences only the high values of the magnetic flux density $|\nabla \hat{u}_{N,M}(\cdot;p)|$ in the magnetic circuit and does not have an impact on the minimum of the reluctivity function. We note that the evaluation of $\delta_{M}(p)$ \eqref{3.3.9} requires one full evaluation of the nonlinearity, thus it is available for the computation in \eqref{lower} for the a-posteriori error estimation.

Once the reduced-basis model is constructed ($N_{max}=12,M_{max}=50)$, we use it to improve the quality of the nonlinearity approximation: we generate the reduced-basis solutions over $\parset_{train}^{EIM(2)}:=\parset_{train}$ and use them to construct the improved EIM approximation space $W_{M}^{\nu}$ of dimension $M_{max}=50$. With the new approximation of the nonlinearity, we run the RB-Greedy procedure over $\parset_{train}$ again with the prescribed tolerance $\epsilon_{RB}=10^{-2}$, which results in the reduced-basis space $\hat{W}_{N}^{u}$ of dimension $N_{max}=10$. 

Next we introduce a parameter test sample $\parset_{test} \subset \parset$ of size 343 ($7 \times 7\times 7$ grid with uniformly random sampling on each interval) and verify the convergence with $N$ of $\max \bigtriangleup_{N,M}=\underset{p \in \parset_{test}}{\max}\bigtriangleup_{N,M}(p)$  for different values of $M$ (see Fig.2(a)). We see that with $N=8$ and $M=50$ the estimator is below the prescribed tolerance $\epsilon_{RB}=10^{-2}$ on the test set. One observes that there is an increase in the estimator for $N \geq 8$ and for $M<50$ due to the poor quality of the EIM approximation. Moreover, we can naturally split the estimator into two parts: the reduced-basis and the nonlinearity approximation error estimation contributions
\begin{align}\label{3.4.1}
\bigtriangleup^{RB}_{N,M}(p):=\frac{\lVert r_{M}(\cdot;p)\rVert_{\hat{X}'}}{\nu_{\text{LB}}\ C_{1}(p)}, \quad \bigtriangleup^{EIM}_{N,M}(p):=\frac{C_{2}(p)\delta_{M}(p)}{\nu_{\text{LB}}\ C_{1}(p)}\lVert \hat{u}_{N,M}(p) \rVert_{\hat{X}}.
\end{align}
The strategy is to balance two contributions in \eqref{3.4.1} for the specified tolerance level $\varepsilon_{RB}$, e.g. (see Fig.2(b)) by choosing $N=8$ and $M=50$ . In Fig. 2(b) we can also see the improvement from the described above additional EIM step. 
\begin{figure}[H]
\centering
\begin{subfigure}{.5\textwidth}
  \centering
  \includegraphics[width=1\textwidth]{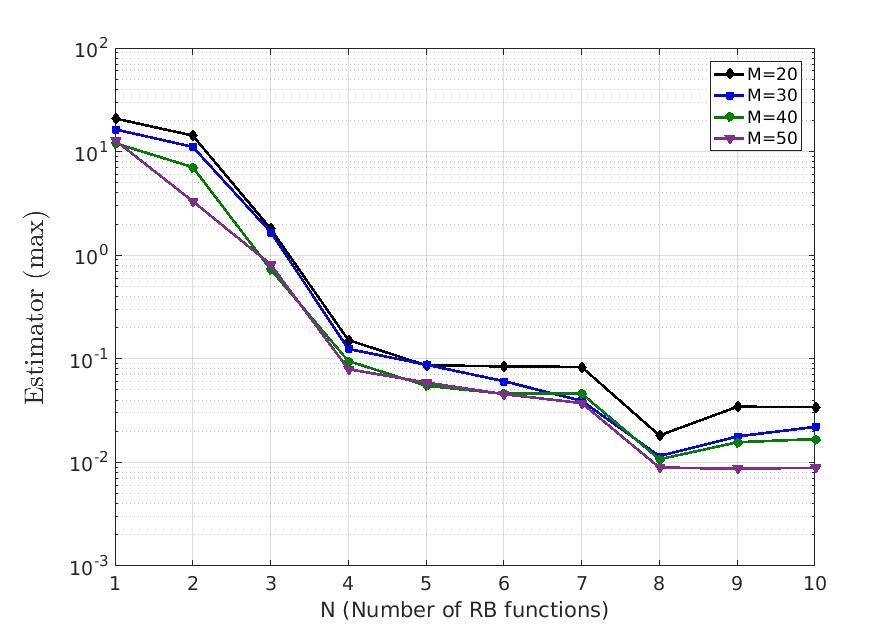}
  \caption{}
  \label{fig:sub1}
\end{subfigure}%
\begin{subfigure}{.5\textwidth}
  \centering
  \includegraphics[width=1\textwidth]{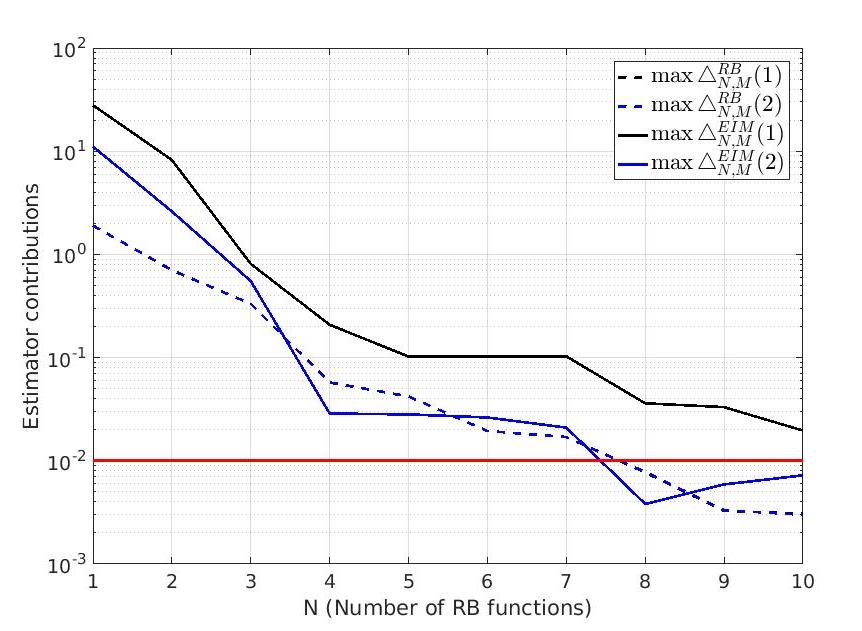}
  \caption{}
  \label{fig:sub2}
\end{subfigure}
\caption{Convergence with $N$ of $\max \bigtriangleup_{N,M}$  for different values of $M$ on the test set (a). Convergence with $N$ of $\max \bigtriangleup^{RB}_{N,M}$ and $\max \bigtriangleup^{EIM}_{N,M}$ contributions for $M=50$ on the test set. The red line is the tolerance level and the number in the label bracket indicates the EIM step (b).}
\end{figure}
In Table 1 we present, as a function of N and M, the maximum error bound $\underset{p \in \parset_{test}}{\max}\bigtriangleup_{N,M}(p)$ as well as the mean $\bar{\eta}_{N,M}$ and $\max \eta_{N,M}$ of the effectivity $\eta_{N,M}(p):=\frac{\bigtriangleup_{N,M}(p)}{\lVert \hat{e}_{N,M}\rVert_{\hat{X}}}$. The effectivities require the knowledge of  ``truth"  solution, therefore we compute the finite element solutions for all the parameters in the test set. We observe that the values of $\bar{\eta}_{N,M}$ and $\max{\eta}_{N,M}$ are quite large, which partially can be explained by the estimate \eqref{3.2.9} for the effectivity $\eta_{N}(p)$ of the reduced-basis approximation. In our example we have 
\begin{align*}
\underset{\hat{x} \in \hat{\Omega}}{\max} \ \nu_{1}(|\nabla \hat{u}_{N,M}(\hat{x};p)|) \leq \nu_{0}    
\end{align*}
on $\parset_{test}$, where $\nu_{0} \approx 7.95 \times 10^5$ is the reluctivity of air. Therefore the upper-bound constant for $\eta_{N}(p)$ is of order $10^{3}$ in practice. 
\begin{table}
\begin{center}
\label{tab:1}  
\caption{Performance of RB-EIM model on the test set}
\begin{tabular}{p{1cm}p{1.5cm}p{2cm}p{2cm} p{2cm} p{2cm}}
\hline\noalign{\smallskip}
$N$ & $M$ & $\max \bigtriangleup_{N,M}$ & $\bar {\bigtriangleup}_{N,M}(p)$ &$ \bar{\eta}_{N,M}$ & $\max{\eta}_{N,M}$\\
\noalign{\smallskip}\svhline\noalign{\smallskip}
4 & 30 & 1.24 E-01 & 4.74 E-02& 7.41 E02 & 1.46 E03\\
6 & 40 & 4.59 E-02 & 2.37 E-02 & 3.98 E02 & 7.18 E02\\
8 & 45 & 9.30 E-03 & 5.10 E-03 & 2.46 E02 & 6.24 E02\\
8 & 50 & 8.90 E-03 & 5.51 E-03 & 2.49 E02 & 6.32 E02\\
10 & 50 & 8.90 E-03 & 5.30 E-03 & 8.48 E02 & 4.65 E03\\
\noalign{\smallskip}\hline\noalign{\smallskip}
\end{tabular}
\end{center}
\end{table}

 In Fig.3 we plot the reduced-basis solutions, i.e. the magnetic equipotential lines for several parameters and the corresponding reluctivity functions, evaluated fully with splines and with EIM. Next we compare the average CPU time required for both the finite element method, which takes $\approx 150$ sec to obtain the solution, and the RB method ($N_{max}=10, M_{max}=50$), which takes $\approx 0.27/0.95$ sec without/with the error bound evaluation and results in the speedup factors of 555 and 158, respectively \footnote{All the computations are performed in MATLAB on Intel Xeon(R) CPU E5-1650 v3, 3.5 GHz x 12, 64 GB RAM}.  The computation of the error bound significantly increases the total CPU time, since the complexity of the error bound evaluation scales quadratically with $Q_{r}$, where $Q_{r}$ is large and requires one full evaluation of the nonlinearity. The offline phase requires the knowledge of the ``truth" finite-element solutions for the first EIM approximation step. Since 216 finite-element solutions were generated in the consecutive order, it takes $\approx$ 9 hours, but it can be done in parallel to reduced the computational time. The Greedy algorithm execution takes $\approx$ 4 hours and since we run it twice, it takes $\approx$ 8 hours for our implementation. We note that our implementation may not be optimal, therefore the offline time is only a rough estimate.    
 
We also note that in the presented numerical example the relatively small parameter domain $\parset$ was chosen. In the authors opinion, it is possible to enlarge the parameter domain with the increasing cost of the nonlinearity approximation by combining few additional EIM steps as described above and exploiting divide-and-conquer principles and hp-adaptivity in the Greedy procedure (see, e.g. \cite{Eftang2012,Sen2008ManyP}).

\section{Conclusion}
\label{sec:4}
In this paper we propose the reduced-basis method for quasilinear elliptic PDEs with application to the nonlinear magnetostatic problem. The geometric parametrisation for the PDE is introduced in the setting of magnet design for the permanent magnet electric motor. We present a new a-posteriori error bound for the class of problems we consider and use it for the weak greedy algorithm and corresponding reduced basis construction. The affine decomposition of the quasilinear form was achieved with the help of EIM. Numerical results confirm a significant speed up factor which supports the validity of the proposed approach.

\begin{figure}[H]
\centering
\begin{subfigure}{.5\textwidth}
  \centering
  \includegraphics[width=1\textwidth]{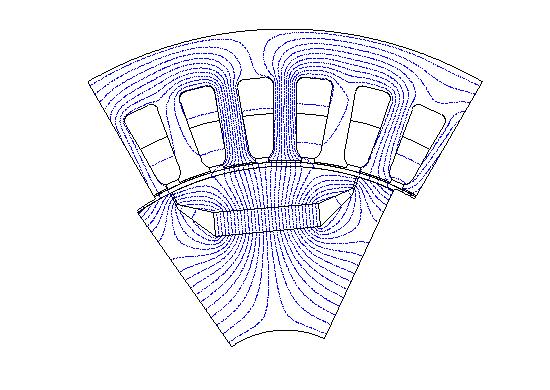}
  \caption{}
  \label{fig:sub1}
\end{subfigure}%
\begin{subfigure}{.5\textwidth}
  \centering
  \includegraphics[width=1\textwidth]{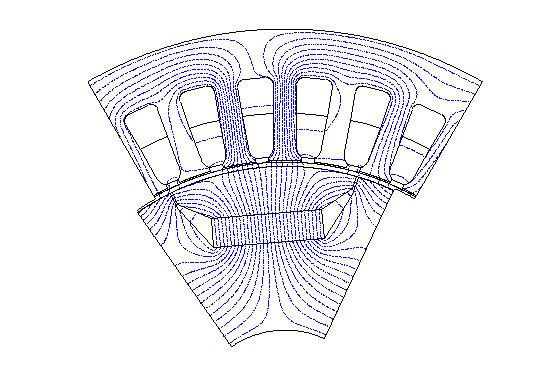}
  \caption{}
  \label{fig:sub2}
\end{subfigure}
\begin{subfigure}{.5\textwidth}
  \centering
  \includegraphics[width=1\textwidth]{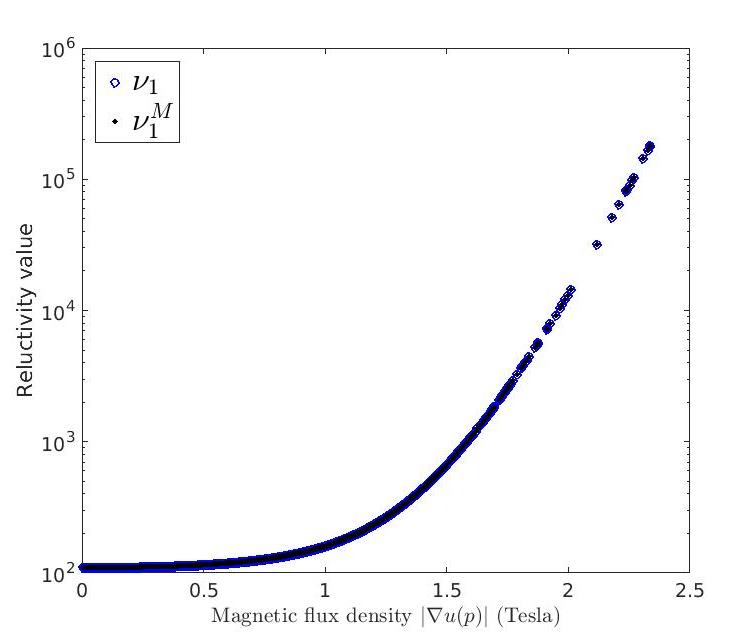}
  \caption{}
  \label{fig:sub1}
\end{subfigure}%
\begin{subfigure}{.5\textwidth}
  \centering
  \includegraphics[width=1\textwidth]{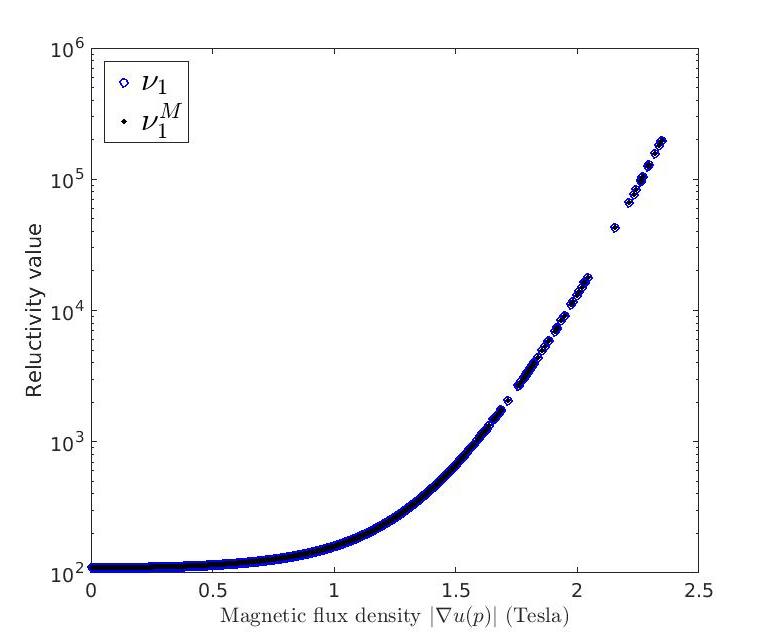}
  \caption{}
  \label{fig:sub1}
\end{subfigure}%
\caption{Magnetic equipotential lines, computed with reduced basis method (10 RB functions, 50 EIM basis functions) for parameter value (a) $p=(18,4,7)$, (b) $p=(19,5,8)$. Reluctivity function $\nu_{1}(p)$, computed with full spline approximation and its EIM counterpart $\nu_{1}^{M}(p)$ for parameter value (c) $p=(18,4,7)$, (d) $p=(19,5,8)$.}
\label{fig:test}
\end{figure}

\section*{Acknowledgement}
Both authors acknowledge the support of the collaborative research project PASIROM funded
by the German Federal Ministry of Education and Research (BMBF) under grant no. 05M2018.

\end{document}